\documentclass[11pt,a4paper]{article}
\usepackage{amscd}
\usepackage{bbm}
\usepackage{mathrsfs}
\usepackage{amsfonts}
\usepackage{pb-diagram}
\usepackage{amsmath}
\usepackage{amssymb}
\usepackage{amsthm}
\usepackage{xcolor}
\makeatletter
\renewcommand\thesection{\@arabic\c@section}
\renewcommand\thesubsection{\thesection.\@arabic\c@subsection}
\makeatother

\newtheorem{theorem}{Theorem}[section]
\newtheorem{lemma}[theorem]{Lemma}
\newtheorem{corollary}[theorem]{Corollary}
\newtheorem{proposition}[theorem]{Proposition}
\newtheorem{definition}{Definition}[section]
\newtheorem{example}{Example}[section]
\newtheorem{remark}{Remark}[section]
\begin{document}
\title{On the first Hochschild cohomology of admissible algebras
\thanks{Project supported by the National Natural Science Foundation of China (No.11271318, No.11171296 and No.J1210038) and the Specialized Research Fund for the Doctoral Program of Higher Education of China (No. 20110101110010)} }
\author{Fang Li\ \ and\ \ Dezhan Tan\\
Department of Mathematics, Zhejiang University\\ Hangzhou, Zhejiang 310027, China\\
fangli@zju.edu.cn;\ \ \ zhijichuanqing@163.com} \maketitle
\begin{abstract}
Our aim in this paper is to investigate the first Hochschild
cohomology of {\em admissible algebras} which can be seen as a
generalization of basic algebras. For this purpose, we study
differential operators on an admissible algebra. Firstly,
differential operators from a path algebra to its quotient algebra
as an admissible algebra are discussed. Based on this discussion,
the first cohomology with admissible algebras as coefficient modules
is characterized, including their dimension formula. Besides, for
planar quivers, the $k$-linear bases of the first cohomology of
acyclic complete monomial algebras and acyclic truncated quiver
algebras are constructed over the field $k$ of characteristic $0$.
\end{abstract}



\section{{ Introduction }}

The Hochschild cohomology of algebras is invariant under Morita
equivalence. Hence it is enough to consider basic connected algebras
when the algebras are Artinian. Let $\Gamma=(V,E)$ be a finite
connected quiver where $V$ (resp. $E$) is the set of vertices (resp.
arrows) in $\Gamma$. Let $k$ be an arbitrary field and $k\Gamma$ be
the corresponding path algebra. Denote by $R$ the two-sided ideal of
$k\Gamma$ generated by $E$. Recall that an ideal $I$ is called {\em
admissible} if there exists $m\geq2$ such that $R^m\subseteq
I\subseteq R^2$ (See \cite{ASS}). According to the Gabriel theorem,
a finite dimensional basic $k$-algebra over an algebraically closed
field $k$ is in the form of $k\Gamma/I$ for a finite quiver $\Gamma$
and an admissible idea $I$.

An Artinian algebra is called a {\em monomial algebra} (see
\cite{ARS}) if it is isomorphic to a quotient $k\Gamma/I$ of a path algebra $k\Gamma$
 for a finite quiver $\Gamma$ and an idea $I$ of
$k\Gamma$ generated by some paths in $\Gamma$. In particular, denote
by $k^n\Gamma$ the ideal of $k\Gamma$ generated by all paths of
length $n$. Then the monomial algebra $k\Gamma/k^n\Gamma$ is called
the {\em $n$-truncated quiver algebra}.

The study of Hochschild cohomology of quiver related algebras
started with the paper of Happel in 1989 \cite{Ha}, who gave the
dimensions of Hochschild cohomology of arbitrary orders of path
algebras for acyclic quivers.  Afterwards, there have been extensive
studies on the Hochschild cohomology of quiver related algebras such
as truncated quiver algebras, monomial algebras, schurian algebras
and 2-nilpotent algebras
\cite{Cibils}\cite{Cib}\cite{Ba}\cite{Z}\cite{L}\cite{St}\cite{Sa}\cite{XHJ}\cite{PS}\cite{ACT}\cite{Sa2}.
In \cite{Ha}, a minimal projective resolution of a finite
dimensional algebra $A$ over its enveloping algebra is described in
terms of the combinatorics when the field $k$ is an algebraically
closed field.  In these papers listed above, the authors use this
kind of  projective resolution or its improving version to compute
the Hochschild cohomology.

 In \cite{GL}, the authors applied an explicit and combinatorial method to study
$HH^1(k\Gamma)$. In this paper, we improve the method in \cite{GL}
to the case of algebras with relations in order to study the
$HH^1(k\Gamma/I)$ where $k\Gamma/I$ is an admissible algebra. This
way does not depend on projective resolution and the requirement of
$k$ being an algebraically closed field. Using this method, we can
obtain some structural results which were not arisen by the classical method in the above listed papers.

If $I\subseteq R^2$ holds for a two-sided ideal $I$, we call
$k\Gamma/I$ an {\em admissible algebra} (see Definition
\ref{admissible algebra}). So finite-dimensional basic algebras are
always admissible algebras. We will give Proposition
\ref{proposition 2.2}, which shows admissible algebras, including
basic algebras, possess the similar characterization of monomial
algebras and truncated quiver algebras, although it is not graded.
From this point of view,  admissible algebra is motivated to unify
and generalize basic algebra and monomial algebra.

In the following, we always assume that $k\Gamma/I$ is an admissible
algebra. This paper includes three sections except for the
introduction. In Section 2, we introduce the basic definitions which
are used in this paper. In particular, we define the notion of an
\emph{acyclic} admissible algebra, which can be thought as a
generalization of the notion of an acyclic quiver. A sufficient and
necessary condition is obtained for a linear operator from $k\Gamma$
to $k\Gamma/I$ to be a differential operator. Next, we give a
standard basis of $\textsl{Diff}(k\Gamma,k\Gamma/I)$.

In Section 3, we investigate $H^1(k\Gamma,k\Gamma/I)$). In Eq.(\ref{3.23}),  a dimension formula of $H^1(k\Gamma,k\Gamma/I)$ is given for a finite dimensional admissible algebra. Moreover, in Theorem
\ref{thm 3.7}, we construct a basis of $H^1(k\Gamma,k\Gamma/I)$ when $\Gamma$ is planar and $k\Gamma/I$ is an acyclic admissible algebra.

In Section 4, we characterize $HH^1(k\Gamma/I)$. In Eq.(\ref{5.28}),
we give the dimension formula of $HH^1(k\Gamma/I)$ for any finite
dimensional admissible algebras $k\Gamma/I$. Moreover, we apply this
method to  complete monomial algebras and truncated quiver algebras.
In Theorems \ref{tm 4.7} and \ref{thm 5.3}, we construct $k$-linear
bases of their first cohomology groups  under certain conditions.
The Hochschild cohomology of monomial algebras and truncated quiver
algebras has been studied in
\cite{Cib}\cite{Z}\cite{L}\cite{XHJ}\cite{S}\cite{XHJ}. Our results
in Section 4 can be seen as the generalization of those
corresponding conclusions in the listed references above. In the
same section,  two examples of admissible algebras are given which
are not monomial algebras. Their first Hochschild cohomology is
characterized using our theory.

\section{$k$-linear basis of $\emph{Diff}(k\Gamma,k\Gamma/I)$ }
We always assume  $\Gamma=(V, E)$ is  a finite connected quiver, where $V$
(resp. $E$) is the set of vertices (resp. arrows) in $\Gamma$. For a
path $p$, denote its starting vertex by $t(p)$, called the
\textbf{tail} of $p$, and the ending point by $h(p)$, called the
\textbf{head} of $p$. For two paths $p$ and $q$, if $t(p)=t(q)$ and
$h(p)=h(q)$, we say $p$ and $q$ to be {\bf parallel}, denote as
$p\parallel q$. Denote by $\mathscr{P}=\mathscr{P}_{\Gamma}$ the set
of paths in a quiver $\Gamma$ including its vertices; denote by
$\mathscr{P}_A$ the set of its acyclic paths. Trivially,  $\Gamma$
is acyclic if and only if $\mathscr{P}_{\Gamma}\backslash
V=\mathscr{P}_A$. Throughout this paper, we always assume quivers
are finite and connected.
\begin{definition}\label{admissible algebra}
Suppose $\Gamma=(V,E)$ is a quiver, $I$ is a
two-sided ideal of $k\Gamma$, we call the quotient algebra
$k\Gamma/I$ an \textbf{admissible algebra} if $I\subseteqq R^2$ where $R$ denotes the two-sided ideal of
$k\Gamma$ generated by $E$.
\end{definition}

\begin{proposition}\label{proposition 2.2}
Suppose $k\Gamma/I$ is an admissible algebra, then there exists a subset $\mathscr{P}^{'}$ of $\mathscr{P}$ such
that $V\cup E\subseteq \mathscr{P}^{'}$ and
$\mathscr{Q}=\{\overline{x}|x\in\mathscr{P}^{'}\}$ forms a basis of $k\Gamma/I$ for $\overline{x}=x+I$
\end{proposition}
\begin{proof}
Let $X$ be a $k$-linear basis of $I$. Denotes by $\mathscr{P}_{\geq2}$  the set of all paths of length $\geq2$. Define $$T:=\{Y\subseteq k\Gamma: Y \text{ is linearly independent in}\ k\Gamma \text{ satisfying}\   X\subseteq Y\subseteq X\cup\mathscr{P}_{\geq2}  \}.$$  $T$ becomes  a partial set due to the order of inclusion between subsets of $k\Gamma$. It is easy to see $T\neq\emptyset$ and $T$ satisfies the upper bound condition of chains. So by the famous Zorn's Lemma, $T$ has a maximal element, denoted by $Z$.

We claim that $Z$ is linearly equivalent to $\mathscr{P}_{\geq2}$. Otherwise, there exists $p\in\mathscr{P}_{\geq2}$ such that $p$ cannot be linearly expressed by $Z$, then $Z\cup\{p\}$ is linearly independent in $k\Gamma$, which contradicts to the maximal property of $Z$.

Since $Z$ is linearly equivalent to $\mathscr{P}_{\geq2}$,  it follows that $V\cup E\cup Z$ is linearly equivalent to $\mathscr{P}=V\cup E\cup\mathscr{P}_{\geq2}$. By the definition of $T$, $Z\subseteq X\cup\mathscr{P}_{\geq2}$. And $I$ is generated by $X$. Hence  $V\cup E\cup (Z\backslash X)$ forms a basis of the complement space of $I$ in $k\Gamma$. It means that $\mathscr Q =\{\bar x: x\in V\cup E\cup (Z\backslash X)\}$  forms a basis of $k\Gamma/I$. It is clear that $V\cup E\cup (Z\backslash X)\subseteq\mathscr{P}$ and it is the $\mathscr{P}^{'}$ we want.
\end{proof}
When $I\subseteq R^2$ is finite dimensional, we have an explicit way to determine the $\mathscr{P}^{'}$. Concretely, suppose $\{x_1,x_2\cdots x_m\}$ is a basis of $I$. Then there exists
a finite subset $\{p_1,p_2\cdots p_n\}$ of $\mathscr{P}$ such that
$x_i$ can be expressed by the linear combinations of $p_j$. Suppose
$x_i=\sum\limits_{j=1}^na_{ij}p_j$ for $i=1,2\cdots m$, then we
obtain a $m\times n$ matrix $A=(a_{ij})$. We can transform the
matrix $A$ into a row-ladder matrix $B=(b_{ij})$ through only row
transformations. Suppose $b_{i,c(i)}$ is the first nonzero number of
the $i$-th row of $B$. Since $B$ is a row-ladder matrix, we have $c_i\neq c_k$
for $i\neq k$. Then $\{x_1,x_2\cdots x_m\}\cup\{p_l|l\neq
c_1,c_2\cdots c_m\}$ is linearly equivalent to $\{p_1,p_2\cdots
p_n\}$. Hence $(\mathscr{P}\backslash\{p_1,p_2\cdots
p_n\})\cup\{p_l|l\neq c_1,c_2\cdots c_m\}$ is a basis of the
complement space of $I$ in $k\Gamma$. Then the residue classes
in $k\Gamma/I$ of all elements in this basis  form a basis of $k\Gamma/I$.

On the other hand, in some special cases, e.g., when $k\Gamma/I$ is a monomial algebra, even if $I$ is not finite dimensional, the choice of $\mathscr{P}'$ is also given in the same way. If $k\Gamma/I$ is a monomial algebra and $I$ is generated by a set of paths of length $\geq2$,  the set of paths that do not belong to $I$ is just the $\mathscr{P}'$ required.
\begin{definition}
Let $A$ be a $k$-algebra and $M$ an $A$-bimodule. A
\textbf{differential operator} (or say, \textbf{derivation}) from
$A$ into $M$ is a $k$-linear map $D: A\longrightarrow M$ such that
\begin{equation}\label{Leibnitz rule}
D(xy)=D(x)y+xD(y).
\end{equation}
In particular, when $M=A$, this coincides with the differential
operator of algebras.
\end{definition}
\begin{lemma}\label{lemma 2.2}
Suppose $D$ is a differential operator from $k\Gamma$ into
$k\Gamma/I$. Then $D$ is determined by its action on the set $V$ of
vertices of $\Gamma$ and the set $E$ of arrows of $\Gamma$.
\end{lemma}
\begin{lemma}\label{2-4}
Let $\Gamma$ be a quiver. Denote $kV$ (resp. $kE$) by the linear
space spanned by the set $V$ of the vertices of $\Gamma$ (resp. the
set $E$ of the arrows of $\Gamma$). Assume we have a pair of linear
maps $D_0: kV\longrightarrow k\Gamma/I$ and $D_1: kE\longrightarrow
k\Gamma/I$ satisfying that
\begin{equation}\label{12}
D_0(x)x+xD_0(x)=D_0(x),\ x\in V,
\end{equation}
\begin{equation}\label{13}
D_0(x)y+xD_0(y)=0,\ x,y\in V,\ x\neq y,
\end{equation}
\begin{equation}\label{14}
D_0(x)q+xD_1(q)=D_1(q),\ x\in V,\ q\in E,\ t(q)=x,
\end{equation}
\begin{equation}\label{15}
D_1(q)y+qD_0(y)=D_1(q),\ y\in V,\ q\in E,\ h(q)=y.
\end{equation}
Then, the pair of linear maps $(D_0,D_1)$ can be uniquely extended
to a differential operator $D: k\Gamma\longrightarrow\ k\Gamma/I$
satisfying that
\begin{equation}\label{16}
D(p):=\sum\limits_{i=1}^lp_1\cdots p_{i-1}D_1(p_i)\cdots p_l.
\end{equation}
for any path $p=p_1p_2\cdots p_l, p_i\in E,1\leq i\leq l, l\geq2$.
\end{lemma}
\begin{proof}
One only need to prove that $D$ is indeed a differential operator.
For this, we need to check  Eq.(\ref{Leibnitz rule}) in the next
four cases:

(a) $x, y\in V$;~ (b) $x\in V, y\in \mathscr{P}\backslash V$; ~ (c)
$x\in\mathscr{P}\backslash V, y\in V$; ~ (d) $x,
y\in\mathscr{P}\backslash V$.

However, the checking process is routine, so we omit it here.
\end{proof}

In the sequel, we always suppose $k\Gamma/I$ is an admissible algebra for the given ideal $I$ and the notations in Definition \ref{admissible algebra} are used. From Definition \ref{admissible algebra} and Proposition \ref{proposition 2.2}, there exists a basis of
$k\Gamma/I$ which consists of residue classes of some paths including
that of $V$ and $E$. Denote the fixed basis of
$k\Gamma/I$ by $\mathscr{Q}$. Suppose $D:k\Gamma\longrightarrow
k\Gamma/I$ is a linear operator, then for any $p\in\mathscr{P}$,
$D(p)$ is a unique combination of the basis $\mathscr{Q}$ of
$k\Gamma/I$. Write this linear combination by
\begin{equation}
D(p)=\sum_{\overline{q}\in\mathscr{Q}}c_{\overline{q}}^p\overline{q}.
\end{equation}
where all $c_{\overline{q}}^p\in k$. We will use this notation
throughout this paper. As convention, for the empty set $\emptyset$,
we say $\sum_{\overline{q}\in\emptyset}c_{\overline{q}}^p\overline{q}=0$.
\begin{lemma}
Suppose $q_1, q_2\in\mathscr{P}, q_1, q_2\notin I$ and
$\overline{q_1}=\overline{q_2}$ in $k\Gamma/I$, then $t(q_1)=t(q_2),
h(q_1)=h(q_2)$, i.e., $q_1\parallel q_2$.
\end{lemma}
\begin{proof}
If $t(q_1)\neq t(q_2)$, then
$\overline{q_1}=\overline{t(q_1)}\overline{q_1}=\overline{t(q_1)}\overline{q_2}=\overline{0}$,
a contradiction, so $t(q_1)=t(q_2)$. Similarly, $h(q_1)=h(q_2)$.
\end{proof}
According to the Lemma above, for $\overline{p}\in\mathscr{Q}$, we can define $t(\overline{p}):=t(q)$(resp.$h(\overline{p}):=h(q)$) for any path $q\in\mathscr{P}$ satisfying $\overline{q}=\overline{p}$ in $k\Gamma/I$. For a path $s\in\mathscr{P}$ and $\overline{p}\in\mathscr{Q}$, if $t(s)=t(\overline{p})$ and $h(s)=h(\overline{p})$, we say $s$ and $\overline{p}$ to be parallel, denoted as $s\parallel\overline{p}$.

Denote $$\mathscr{Q}_A:=\{\overline{p}\in\mathscr{Q}|t(\overline{p})\neq
h(\overline{p})\}\  \text{ and}  \ \
\mathscr{Q}_C:=\{\overline{p}\in\mathscr{Q}|t(\overline{p})=
h(\overline{p})\}.$$ Moreover
$k\mathscr{Q}_A$(resp.,$k\mathscr{Q}_C$) denotes the subspace of
$k\Gamma/I$ generated by $\mathscr{Q}_A$(resp.,$\mathscr{Q}_C$).
Clearly,  as $k$-linear spaces,
$k\Gamma/I=k\mathscr{Q}_A\oplus k\mathscr{Q}_C$.
\begin{definition} Using the above notations,
an admissible algebra $k\Gamma/I$ is called  {\bf acyclic} if $$\mathscr{Q}_C\backslash\{\overline{v}|v\in V\}=\emptyset.$$
\end{definition}
It is easy to see from this definition that \\
(i)~ The fact whether the given $k\Gamma/I$ is acyclic is independent with  the
choice of $\mathscr{Q}$.\\
(ii)~ If the quiver $\Gamma$ is acyclic, then $k\Gamma/I$ is acyclic; the converse is not true in general.\\
(iii)~ If $k\Gamma/I$ is acyclic, then it is finite dimensional; the converse is not true, e.g., $k\Gamma/k^n\Gamma$ if $\Gamma$ is a loop for $n\geq2$.

\begin{proposition}\label{2-5}
Let $D: k\Gamma\longrightarrow k\Gamma/I$ be a $k$-linear operator.

(i)~ If $D$ is a differential operator, then

(a) for $v\in V$,
\begin{equation}\label{19}
D(v)=\sum_{\overline{q}\in\mathscr{Q}, t(\overline{q})=v,
h(\overline{q})\neq v}c_{\overline{q}}^{v}\overline{q}+
\sum_{\overline{q}\in\mathscr{Q}, h(\overline{q})=v,
t(\overline{q})\neq v}c_{\overline{q}}^{v}\overline{q},
\end{equation}

(b) for $p\in E$,
\begin{equation}\label{20}
D(p)=\sum_{\substack{\overline{q}\in\mathscr{Q},
\\h(\overline{q})=t(p),t(\overline{q})\neq
t(p)}}c_q^{t(p)}\overline{qp}+\sum_{\substack{\overline{q}\in\mathscr{Q},
\\\overline{q}\parallel p}}c_{\overline{q}}^{p}\overline{q}+\sum_{\substack{\overline{q}\in\mathscr{Q},
\\t(\overline{q})=h(p),h(\overline{q})\neq h(p)}}c_{\overline{q}}^{h(p)}\overline{pq}
\end{equation}
where the coefficients are subject to the following condition: for
any path $\overline{q}\in \mathscr{Q}$ such that
$t(\overline{q})\neq h(\overline{q})$,
\begin{equation}\label{21}
c_{\overline{q}}^{h(\overline{q})}+c_{\overline{q}}^{t(\overline{q})}=0.
\end{equation}

(ii)~ Conversely, assume the linear map $D$ from $kV\oplus kE$ to
$k\Gamma/I$ satisfies Eqs.(\ref{19}), (\ref{20}), (\ref{21}), then
$D$ can be uniquely extended linearly to a differential operator as
Eq.(\ref{16}).
\end{proposition}
\begin{proof}
(i) For a given $v\in V$, since $vv=v$, we have
$$D(v)=D(vv)=D(v)v+v D(v).$$
So by the direct computation, we can get
\begin{eqnarray}
D(v)&=&\sum_{\overline{q}\in \mathscr{Q},
h(\overline{q})=v}c_{\overline{q}}^{v}\overline{q}+\sum_{\overline{q}\in
\mathscr{Q}, t(\overline{q})=v}c_{\overline{q}}^{v} \overline{q}\notag.\\
\end{eqnarray}
Moreover,
$$D(v)=D(v)v+v D(v)=(D(v)v+v D(v))v+v D(v)=D(v)v+v D(v)v+v D(v),$$
so we have $v D(v)v=0$. That means  $\sum\limits_{\overline{q}\in
\mathscr{Q},
t(\overline{q})=h(\overline{q})=v}c_{\overline{q}}^{v}\overline{q}=0$.
So we get Eq.(\ref{19}).

Also, for a given $p\in E$, we have
\begin{eqnarray*}
D(p)&=&D(t(p)ph(p))\\
&=&D(t(p)){ph(p)}+t(p)D(p)h(p)+t(p)p D(h(p))\\
&=&D(t(p))p+\sum_{\substack{\overline{q}\in\mathscr{Q}\\\overline{q}\parallel p}}c_{\overline{q}}^p\overline{q}+p D(h(p))\\
\end{eqnarray*}
Since $t(p),h(p)\in V$, by Eq.(\ref{19}), we can easily get
Eq.(\ref{20}).

Let $x,y\in V,\ x\neq y$. By (\ref{19}),
\begin{eqnarray*}
D(xy)&=&D(x)y+xD(y)\\
&=&\sum\limits_{\overline{q}\in\mathscr{Q},t(\overline{q})=x,h(\overline{q})=y}c_{\overline{q}}^{x}\overline{q}+\sum\limits_{\overline{q}\in\mathscr{Q}, t(\overline{q})=x,h(\overline{q})=y}c_{\overline{q}}^{y}\overline{q}\\
&=&\sum\limits_{\overline{q}\in\mathscr{Q},
t(\overline{q})=x,h(\overline{q})=y}(c_{\overline{q}}^{x}+c_{\overline{q}}^{y})\overline{q}
\end{eqnarray*}
But, $D(xy)=0$ since $xy=0$. So,
$\sum\limits_{\overline{q}\in\mathscr{Q},
t(\overline{q})=x,h(\overline{q})=y}(c_{\overline{q}}^{x}+c_{\overline{q}}^{y})\overline{q}=0$.

For a path $\overline{q}\in\mathscr{Q}$ such that $t(\overline{q})\neq h(\overline{q})$, substituting $x$ and $y$
respectively with $t(\overline{q})$ and $h(\overline{q})$, we get
Eq.(\ref{21}).
\\

(ii)~ We only need to verify the conditions of Lemma \ref{2-4} are
satisfied. Because the process is straightforward, we leave it to
the readers.
\end{proof}

Next, we apply Proposition \ref{2-5} to display a standard
basis of differential operators from $k\Gamma$ to $k\Gamma/I$, for any admissible algebra
$k\Gamma/I$.
\begin{proposition} (\textbf{Differential operator} $D_{r,\overline{s}}$.)
For a quiver $\Gamma=(V, E)$, let $r\in E$ and $s\in \mathscr{P}$
with $r\parallel s$. Define the $k$-linear operator
$D_{r,\overline{s}}:{k}V\oplus{k}E\longrightarrow k\Gamma/I$
satisfying
\begin{equation}\label{29}
D_{r,\overline{s}}(p)=\begin{cases}
\overline{s}, &p=r~~\text{for}~~ p\in E,\\
0, &p\neq r~~\text{for}~~ p\in E\cup V, \end{cases}
\end{equation}
Then, the conditions of Lemma(\ref{2-4}) are satisfied and thus,
$D_{r,\overline{s}}$ can be uniquely extended to a differential
operator from ${k}\Gamma$ to $k\Gamma/I$, denoted still by
$D_{r,\overline{s}}$ for convenience.
\end{proposition}
\begin{proof}
Eqs.(\ref{12}), (\ref{13}), Eq.(\ref{14}) and Eq.(\ref{15}) can be
checked easily by the definition of $D_{r,\overline{s}}$.
\end{proof}

 For a given $s\in \mathscr{P}$,
we have the corresponding inner differential operator:
\begin{equation}\label{2.13}
D_{\overline{s}}: {k}\Gamma\rightarrow k\Gamma/I,\
D_{\overline{s}}(q)=\overline{sq}-\overline{qs},\ \forall
q\in\mathscr{P}.
\end{equation}
\begin{theorem}\label{theorem 3.2}
Let $\Gamma=(V,E)$ be a quiver and $I$ be an ideal such that $k\Gamma/I$ is an admissible algebra. Then  the set
\begin{equation}
\mathfrak{B}\:=\mathfrak{B}_1\cup\mathfrak{B}_{2}
\end{equation}
is a basis of the $k$-linear space of differential operators
from $k\Gamma$ to $k\Gamma/I$,
where \begin{equation}\label{eq2.15}
\mathfrak{B}_1:=\{D_{\overline{s}}|\overline{s}\in\mathscr{Q}_A\},\ \ \
\mathfrak{B}_{2}:=\{D_{r,\overline{s}}|r\in
E,\overline{s}\in\mathscr{Q}, r\parallel\overline{s} \}.
\end{equation}
\end{theorem}
\begin{proof} We only need to verify that the operators in
$\mathfrak{B}$ are linearly independent and any differential
operators can be generated $k$-linearly by $\mathfrak{B}$.

 \textbf{Step 1. $\mathfrak{B}$ is linearly
independent.} Suppose there are
$c_{\overline{p}},c_{r,\overline{s}}\in k$ such that
\begin{equation}\label{independent}
\sum_{\overline{p}\in\mathscr{Q}, h(\overline{p})\neq
t(\overline{p})}c_{\overline{p}}D_{\overline{p}}+\sum_{r\in
E,\overline{s}\in \mathscr{Q},r\parallel
\overline{s}}c_{r,\overline{s}}D_{r,\overline{s}}=0.
\end{equation}

Then for any given $\overline{p_0}\in \mathscr{Q},
h(\overline{p_0})\neq t(\overline{p_0})$, by the definition of
$D_{\overline{p}}$ and $D_{r,\overline{s}}$, we have

\begin{eqnarray*}
0&=&\sum_{\overline{p}\in\mathscr{Q}, t(\overline{p})\neq
h(\overline{p})}c_{\overline{p}}D_{\overline{p}}(h(\overline{p_0}))+\sum_{r\in
E,\overline{s}\in
\mathscr{Q}, r\parallel \overline{s}}c_{r,\overline{s}}D_{r,\overline{s}}(h(\overline{p_0}))\\
&=&\sum_{\overline{p}\in\mathscr{Q}, t(\overline{p})\neq h(\overline{p})}c_{\overline{p}}(\overline{ph(p_0})-\overline{h(p_0)p})+0\\
&=&\sum_{\overline{p}\in\mathscr{Q},t(\overline{p})\neq
h(\overline{p})=h(\overline{p_0})}c_{\overline{p}}\overline{p}-\sum_{\overline{q}\in\mathscr{Q},h(\overline{q})\neq
t(\overline{q})=h(\overline{p_0})}c_{\overline{q}}\overline{q}.
\end{eqnarray*}
In the last formula above, $\overline{p}\ne \overline{q}$ always
holds. Thus, their coefficients are all zero. In particular,
$c_{\overline{p_0}}=0$ for any $\overline{p_0}\in\mathscr{Q}$ with
$h(\overline{p_0})\neq t(\overline{p_0})$.

Thus, from (\ref{independent}), we get that
$$\sum_{r\in E,\overline{s}\in
\mathscr{Q}, r\parallel
\overline{s}}c_{r,\overline{s}}D_{r,\overline{s}}=0.$$
Further, for any given $r_0\in E, \overline{s}\in \mathscr{Q}$ with
$\overline{s}\parallel r_0$, we have:

$\sum\limits_{r\in E,\overline{s}\in \mathscr{Q},r\parallel
\overline{s}}c_{r,\overline{s}}D_{r,\overline{s}}(r_0)=0~~~~\Longrightarrow~~~~\sum\limits_{\overline{s}\in
\mathscr{Q},r_0\parallel \overline{s}}c_{r_0,\overline{s}}\overline{s}=0$.\\
It follows that $c_{r_0,\overline{s}}=0$ for any $r_0\in E,
\overline{s}\in \mathscr{Q}$ with $r\parallel \overline{s}$.

Hence,  $\mathfrak{B}$ is $k$-linearly independent.

\textbf{Step 2. $\mathfrak{B}$ is the set of $\bf k$-linear
generators.} Let $D: {k}\Gamma\rightarrow k\Gamma/I$ be any
differential operator. Then for $v\in V$ and $p\in E$, by
Eqs.(\ref{19}), (\ref{20}) and (\ref{21}) we have
\begin{equation}\label{033}
D(v)=\sum_{\overline{q}\in\mathscr{Q}, h(\overline{q})\neq
t(\overline{q})=v}c_{\overline{q}}^{v}\overline{q}+
\sum_{\overline{q}\in\mathscr{Q},t(\overline{q})\neq
h(\overline{q})=v}c_{\overline{q}}^{v}\overline{q}.
\end{equation}
\begin{equation}\label{33}
D(p)=-\sum_{\substack{\overline{q}\in\mathscr{Q}\\t(\overline{q})\neq
h(\overline{q})=t(p)}}c_{\overline{q}}^{t(\overline{q})}\overline{qp}+\sum_{\substack{\overline{q}\in\mathscr{Q}\\\overline{q}\parallel
p}}c_{\overline{q}}^p\overline{q}+
\sum_{\substack{\overline{q}\in\mathscr{Q}\\h(\overline{q})\neq
t(\overline{q})=h(p)}}c_{\overline{q}}^{t(\overline{q})}\overline{pq}.
\end{equation}
We claim that $D$ agrees with the differential operator
$\overline{D}$ defined by the linear combination
\begin{equation}\label{III.23}
\overline{D}=-\sum_{\overline{s}\in\mathscr{Q}, t(\overline{s})\neq
h(\overline{s})}c_{\overline{s}}^{t(\overline{s})}D_{\overline{s}}+\sum_{r\in
E,\overline{s}\in\mathscr{Q}, \overline{s}\parallel
r}c_{\overline{s}}^rD_{r,\overline{s}},
\end{equation}
where $c_{\overline{s}}^{t(\overline{s})}$ and $c_{\overline{s}}^r$
come from Eqs.(\ref{033}) and (\ref{33}). Any path in $\mathscr{P}$
is either a vertex or a product of arrows. Thus by the product rule
of differential operators, to show
 $D=\overline{D}$, we only need to
verify that $D(q)=\overline{D}(q)$ for each $q=v\in V$ and $q=p\in
E$. The verification is straightforward, so we omit it.
\end{proof}
We call the set $\mathfrak{B}$ in Theorem \ref{theorem 3.2} the \textbf{standard basis} of the $k$-linear space $\emph{Diff}(k\Gamma, k\Gamma/I)$ generated by all differential operators from $k\Gamma$ to $k\Gamma/I$.

From this theorem, we get $\textsl{Diff}(k\Gamma,k\Gamma/I)=\mathfrak{D}_1\oplus
\mathfrak{D}_{2}$, where $\mathfrak{D}_i$ is the $k$-linear space
generated by $\mathfrak{B}_i$ for $i=1,2$ in (\ref{eq2.15}).

For any $p\in E$,
$D_{p,\overline{p}}\in\mathfrak{B}_{2}$ is called  \textbf{arrow differential operator} from
$k\Gamma$ to $k\Gamma/I$. Let $\mathfrak{B}_E:=\{D_{p,\overline{p}}|p\in E\}$ and $\mathfrak{D}_E:=k\mathfrak{B}_E$ is called the \textbf{space of arrow  differential operators}.

\begin{remark}
When $r\in E$ is a loop of $\Gamma$, i.e., $t(r)=h(r)$, then
$D_{r,\overline{t(r)}}\in\mathfrak{B}_{2}$.
\end{remark}

\section{$H^1(k\Gamma,k\Gamma/I)$ for an admissible algebra $k\Gamma/I$}
\begin{proposition}\label{proposition 4.1}
Let $q\in\mathscr{P}$ be such that $h(q)=t(q)=v_0$. We have
\begin{equation}\label{4.21}
D_{\overline{q}}=\sum\limits_{p\in E,
t(p)=v_0}D_{p,\overline{qp}}-\sum\limits_{r\in E,
h(r)=v_0}D_{r,\overline{rq}}.
\end{equation}
\end{proposition}
\begin{proof}
Note that the both sides of (\ref{4.21}) are $k$-linearly generated
by differential operators. So, by the product formula of
differential operators, we only need to verify that the both sides
always agree when they act on the elements of $V$ and $E$. Since the
computation is direct, we omit it here.
\end{proof}
\begin{remark}
For $v\in V$, it is clear that $t(v)=h(v)=v$. From Proposition
\ref{proposition 4.1}, we have
\begin{equation}
D_{\overline{v}}=\sum\limits_{p\in E,
t(p)=v}D_{p,\overline{p}}-\sum\limits_{r\in E,
h(r)=v}D_{r,\overline{r}}.
\end{equation}
We call $D_{\overline{v}}$ the \textbf{vertex differential operator}
from $k\Gamma$ to $k\Gamma/I$. Let $\mathfrak{D}_{V}$ denote the linear space spanned by
$\{D_{\overline{v}}|v\in V\}$, called the \textbf{space of vertex differential operators}. It is clear that
$\mathfrak{D}_{V}$ is a subspace of $\mathfrak{D}_{E}$.
\end{remark}
\begin{lemma}\label{lemma 4.2}
Let $p\in \mathscr{P}$, then $\overline{p}$ is always in the $k$-subspace $k\{\overline{q}\in
\mathscr{Q}|\ \overline{q}\parallel p\}$ generated by $\overline{q}$ with $\overline{q}\parallel p$.
\end{lemma}
\begin{proof}
Suppose $\overline{p}=\sum\limits_{\overline{q}\in
\mathscr{Q}}c_{\overline{q}}\overline{q}$, then
$\overline{t(p)ph(p)}=\sum\limits_{\overline{q}\in
\mathscr{Q}}c_{\overline{q}}\overline{t(p)qh(p)}=\sum\limits_{\substack{\overline{q}\in
\mathscr{Q}\\\overline{q}\parallel p}}c_{\overline{q}}\overline{q}.$
\end{proof}
\begin{corollary}\label{corollary 4.3}
Let $q\in\mathscr{P}$ be such that $h(q)=t(q)$. Then
$D_{\overline{q}}\in k\mathfrak{B}_{2}=\mathfrak{D}_{2}$.
\end{corollary}
\begin{proof}
For $r\in E, r\parallel s\in \mathscr{P}$; by Lemma \ref{lemma 4.2},
suppose $\overline{s}=\sum\limits_{\substack{\overline{q}\in
\mathscr{Q}\\\overline{q}\parallel s}}c_{\overline{q}}\overline{q}$,
and it is clear that
$D_{r,\overline{s}}=\sum\limits_{{\substack{\overline{q}\in
\mathscr{Q}\\\overline{q}\parallel
s}}}c_{\overline{q}}D_{r,\overline{q}}$, then use Proposition
\ref{proposition 4.1}.
\end{proof}

\begin{remark}\label{remark 4.1}
For $\overline{q}\in\mathscr{Q}, t(\overline{q})=h(\overline{q})$,
from Theorem \ref{theorem 3.2} and Corollary \ref{corollary 4.3}, we
know that $D_{\overline{q}}\in k\mathfrak{B}_{2}=\mathfrak{D}_{2}$,
but not in $k\mathfrak{B}_1=\mathfrak D_{1}$. Denote
$\mathfrak{D}_C:=k\{D_{\overline{q}}\ |\
\overline{q}\in\mathscr{Q},\ t(\overline{q})= h(\overline{q})\}$.
Then $\mathfrak{D}_{C}\subseteq\mathfrak{D}_{2}$ and
$\mathfrak{D}_{C}\cap \mathfrak{D}_{1}=0$.
\end{remark}

Denote by $\textsl{Inn-Diff}(k\Gamma, k\Gamma/I)$  the linear
space consisting of inner differential operators from $k\Gamma$ to $k\Gamma/I$. Then,  $\textsl{Inn-Diff}(k\Gamma, k\Gamma/I)=\mathfrak{D}_{1}+\mathfrak{D}_{C}$. Thus, we have
\begin{eqnarray*}
H^1(k\Gamma, k\Gamma/I)&=&\textsl{Diff}(k\Gamma, k\Gamma/I)/\textsl{Inn-Diff}(k\Gamma, k\Gamma/I)\\
&=&(\mathfrak{D}_{1}+\mathfrak{D}_{2})/(\mathfrak{D}_{1}+\mathfrak{D}_{C})\\
&\cong&\mathfrak{D}_{2}/(\mathfrak{D}_{2}\cap \mathfrak{D}_{C})\\
&\cong &\mathfrak{D}_{2}/ \mathfrak{D}_{C}.
\end{eqnarray*}
Since the basis of $k\Gamma/I$ given in Proposition \ref{proposition 2.2} contains the residue classes of $V$ and $E$, we can see that
the center of $k\Gamma/I$ as $k\Gamma$-bimodule and the center of
$k\Gamma/I$ as an algebra are the same, denoted by $Z(k\Gamma/I)$.
\begin{proposition}
Let $k\Gamma/I$ be a finite dimensional admissible algebra, then
\begin{equation}\label{3.23}
dim_kH^1(k\Gamma,k\Gamma/I)=|\mathfrak{B}_2|+dim_kZ(k\Gamma/I)-|\mathscr{Q}_C|.
\end{equation}

\end{proposition}
\begin{proof}By the discussion above,
$dim_kHH^1(k\Gamma,k\Gamma/I)=|\mathfrak{B}_2|-dim_k\mathfrak{D}_{C}$.
And\begin{eqnarray*}
\mathfrak{D}_1\oplus\mathfrak{D}_C=\textsl{Inn-Diff}(k\Gamma, k\Gamma/I)&\cong&(k\Gamma/I)/Z(k\Gamma/I)\\
&\cong&(k\mathscr{Q}_C\oplus k\mathscr{Q}_A)/Z(k\Gamma/I)\\
&\cong&k\mathscr{Q}_C/(Z(k\Gamma/I))\oplus k\mathscr{Q}_A\\
&\cong&k\mathscr{Q}_C/(Z(k\Gamma/I))\oplus \mathfrak{D}_1,
\end{eqnarray*}
where the first isomorphism is assured by Eq.(\ref{2.13}), the second and fourth isomorphisms are trivial, the third is because of the facts that $Z(k\Gamma/I)\subseteq k\mathscr{Q}_C$ and $Z(k\Gamma/I)\cap k\mathscr{Q}_A=0$. So
$\mathfrak{D}_C\cong k\mathscr{Q}_C/Z(k\Gamma/I)$ as $k$-linear spaces, it follows that
\begin{equation}
dim_kH^1(k\Gamma,k\Gamma/I)=dim_k\mathfrak{D}_{2}-dim_k\mathfrak{D}_C=|\mathfrak{B}_2|+dim_kZ(k\Gamma/I)-|\mathscr{Q}_C|.
\end{equation}
\end{proof}

If $k\Gamma/I$ is acyclic, then $Z(k\Gamma/I)\cong k$ and $|\mathscr{Q}_C|=|V|$. Thus, we have

\begin{corollary}
If $k\Gamma/I$ is an acyclic admissible algebra (in particular, if $\Gamma$ is an acyclic quiver), then
\begin{equation}
dim_kH^1(k\Gamma,k\Gamma/I)=|\mathfrak{B}_2|+1-|V|.
\end{equation}
\end{corollary}

On the other hand, when $\Gamma$ is a planar quiver and $k\Gamma/I$
is an acyclic admissible algebra, we can apply the approach of \cite{GL} to give a basis
of $HH^1(k\Gamma,k\Gamma/I)$. A planar quiver is a quiver with a
fixed embedding into the plane $\mathbb{R}^2$. The set $F$ of faces
of a planar quiver $\Gamma$ is the set of connected component of
$\mathbb{R}^2\backslash\Gamma$.

We will need the famous {\em Euler formula} on planar graph, see \cite{B}\cite{GT}, which states that for any finite connected planar graph (which can be thought as the underlying graph of a quiver $\Gamma$), we have
\begin{equation}\label{4.22}
|V|-|E|+|F|=2.
\end{equation}

For each face of $\Gamma$, its boundary is called a
\textbf{primitive cycle}. Let $\mathbbm{p}_0$ denote the boundary of
the unique unbounded face $f_0$ of $\Gamma$. Let $\Gamma_{\mathbb{P}}$
denote the set of primitive cycles of $\Gamma$ and
$\Gamma_{\mathbb{P}}^{-}:=\Gamma_{\mathbb{P}}\backslash\mathbbm{p}_0$.
Then clearly, the set $\Gamma_{\mathbb{P}}$ of primitive cycles of
$\Gamma$ is in bijection with the set $F$ of the faces of $\Gamma$.
So $|F|=|\Gamma_{\mathbb{P}}|$.

For a face $f\in F$, denote $\mathbbm{p}_f$ the corresponding
primitive cycle of $f$. Suppose $\mathbbm{p}_f$ is comprised of an
ordered list arrows $p_1,\cdots,p_s\in E$, define an operator from
$k\Gamma$ to $k\Gamma/I$
\begin{equation}
D_{\mathbbm{p}_f}:=\pm D_{p_1,\overline{p_1}}\pm\cdots\pm
D_{p_s,\overline{p_s}},
\end{equation}
where a $\pm D_{p_i,\overline{p_i}}$ is $+D_{p_i,\overline{p_i}}$ if
$p_i$ is in clockwise direction when viewed from the interior of the
face of $\mathbbm{p}_f$ and is $-D_{p_i,\overline{p_i}}$ otherwise.
We call $D_{\mathbbm{p}_f}$ a \textbf{face differential operator}
from $k\Gamma$ to $k\Gamma/I$. Let $\mathfrak{D}_{\mathbbm{P}}$ denote the linear space spanned by $\{D_{\mathbbm{P}}|\mathbbm{p}\in\Gamma_{\mathbb{P}}\}$, called the \textbf{space of face differential operators}.

The next lemma is similar to Theorem 4.9 in \cite{GL}.
\begin{lemma}\label{lemma 4.5}
Let $\Gamma$ be a planar quiver with the ground field $k$
of characteristic 0, then

(a) $dim\mathfrak{D}_V=|V|-1$;

(b) $dim\mathfrak{D}_{\mathbb{P}}=|F|-1=|\Gamma_{\mathbb{P}}^{-}|$;

(c) $\mathfrak{D}_V$ and $\mathfrak{D}_{\mathbb{P}}$ are linearly
disjoint subspaces of $\mathfrak{D}_E$.
\end{lemma}
\begin{proof}
(a) Denote $\gamma_0=|V|$. Since $\overline{e}=\sum\limits_{i=1}^{\gamma_0}\overline{v_i}$ is the identity of $k\Gamma/I$, which clearly lies in the center of $k\Gamma/I$, we have
\begin{equation}
D_{\overline{e}}=\sum\limits_{i=1}^{\gamma_0}D_{\overline{v_i}}=0.
\end{equation}
So $dim\mathfrak{D}_V\leq\gamma_0-1$. We next prove that $dim\mathfrak{D}_V\geq\gamma_0-1$. We may assume that $\gamma_0\geq2$.

We claim that  any $\gamma_0-1$ elements of $\{D_{\overline{v_i}}|i=1,\cdots,\gamma_0\}$ is linearly independent. In fact, suppose
$\sum\limits_{i=1}^{\gamma_0-1}a_iD_{\overline{v_i}}=0$,
where $a_i\in k$, which  means that $\sum\limits_{i=1}^{\gamma_0-1}a_i\overline{v_i}$ is in the center of $k\Gamma/I$. Since $\Gamma$ is connected, let the vertex $v_{\gamma_0}$ be connected to $v_i$ by an arrow $p$  for $i\neq\gamma_0$. We may assume that $t(p)=v_i$ and $h(p)=v_{\gamma_0}$. We have $$a_i\overline{p}=(\sum\limits_{i=1}^{\gamma_0-1}a_i\overline{v_i})\overline{p}=\overline{p}(\sum\limits_{i=1}^{\gamma_0-1}a_i\overline{v_i})=\overline{0}.$$ so $a_i=0$. Note that $\Gamma$ is connected, we can repeat this process to get $a_j=0$ for any $j$.

(b) Let $|F|=\gamma_2$. Through simple observation of planar quiver, we can see that if $p\in E$ is in the boundary, then it is at most in the boundary of two primitive cycles. Note that if $p\in E$ is in the boundary of two primitive cycles $\mathbbm{P}_1$ and $\mathbbm{P}_2$, then the sign of $D_{p,\overline{p}}$ in $D_{\mathbbm{P}_1}$ and $D_{\mathbbm{P}_2}$ are opposite. If $p\in E$ is in the boundary of only one primitive cycle $\mathbbm{P}$, then $D_{p,\overline{p}}$ occurs twice in $D_{\mathbbm{P}}$ with opposite sign. Thus we have
\begin{equation}
\sum\limits_{j=0}^{\gamma_2-1}D_{\mathbbm{P}_i}=0,
\end{equation}
where $\mathbbm{P}_0$ denotes the primite cycle corresponding to $f_0$ as above. So $dim\mathfrak{D}_{\mathbb{P}}\leq|F|-1$.

We next prove that $dim\mathfrak{D}_{\mathbb{P}}\geq|F|-1$. We may assume that $|F|\geq2$. Suppose
\begin{equation}
\sum\limits_{j=1}^{\gamma_2-1}b_jD_{\mathbbm{P}_j}=0,
\end{equation}
where $b_j\in k$. If $p\in E$ is in the boundary of $\mathbbm{P}_0$ and $\mathbbm{P}_j$ for $j>0$, then $\overline{0}=\sum\limits_{j=1}^{\gamma_2-1}b_jD_{\mathbbm{P}_j}(p)=\pm b_j\overline{p}$. So we have $b_j=0$. This means that if $\mathbbm{P}_j$ and $\mathbbm{P}_0$ have a common $p\in E$ in their boundary, then $b_j=0$. Replace $\mathbbm{P}_0$ with $\mathbbm{P}_j$, and repeat this process. Since the quiver is connected, we can get $b_j=0$ for any $j>0$.

(c)  From \cite{GL} and Theorem \ref{theorem 3.2}, we know that
$\mathfrak{B}^o_E:=\{D_{p,p}| p\in E\}$ and
$\mathfrak{B}_E:=\{D_{p,\overline{p}}| p\in E\}$ are $k$-linearly
independent sets in $\textsl{Diff}(k\Gamma)$ and
$\textsl{Diff}(k\Gamma, k\Gamma/I)$ respectively. Based on this,
$D_{\overline{v_i}}$ and $D_{\mathbbm{p}_f}$ can be linearly
expressed by using  $\mathfrak{B}_E$, as well as $D_{v_i}$ and
$D_{\mathfrak{c}_f}$ by using $\mathfrak{B}^o_E$ in  \cite{GL}.
Under this  correspondence, referring to Theorem 4.9 in \cite{GL} in
the same process, we obtain that  $\mathfrak{D}_V$ and
$\mathfrak{D}_{\mathbb{P}}$ are linearly disjoint subspaces of
$\mathfrak{D}_E$.
\end{proof}

By this lemma,  $\mathfrak{B}_{\mathbbm{P}}:=\{D_{\mathbbm{P}}|\mathbbm{p}\in\Gamma_{\mathbb{P}}^{-}\}$ is a basis of $\mathfrak{D}_{\mathbbm{P}}$.
\begin{theorem}\label{thm 3.7}
Let $\Gamma$ be a planar quiver and $k\Gamma/I$ be an acyclic admissible algebra with the ground field $k$ of characteristic 0.
Then the union set
$$(\mathfrak{B}_{2}\backslash\mathfrak{B}_E)\cup\mathfrak{B}_{\mathbb{P}}$$
is a basis of $H^1(k\Gamma,k\Gamma/I)$.
\end{theorem}
\begin{proof}
By the Euler formula and  Lemma \ref{lemma 4.5}, we can get
$\mathfrak{D}_E=\mathfrak{D}_V\oplus\mathfrak{D}_{\mathbb{P}}$.
Because $k\Gamma/I$ is acyclic, we have $\mathfrak{D}_{C}=\mathfrak{D}_{V}$, then
\begin{eqnarray*}
H^1(k\Gamma, k\Gamma/I)&\cong &\mathfrak{D}_{2}/
\mathfrak{D}_{C}\\
&\cong &(\mathfrak{D}_E\oplus
k\{\mathfrak{B}_{2}\backslash\mathfrak{B}_E\})/\mathfrak{D}_{V}\\
&\cong&\mathfrak{D}_{\mathbb{P}}\oplus k\{\mathfrak{B}_{2}\backslash\mathfrak{B}_E\}\\
&\cong&k\mathfrak{B}_{\mathbbm{P}}\oplus
k\{\mathfrak{B}_{2}\backslash\mathfrak{B}_E\}.
\end{eqnarray*}\end{proof}

\section{$HH^1(k\Gamma/I)$ for an admissible algebra $k\Gamma/I$ }
\begin{lemma}\label{lemma 5.1}
A differential operator of $k\Gamma/I$ can induce naturally a
differential operator from $k\Gamma$ to $k\Gamma/I$. Conversely, a
differential operator $D$ from $k\Gamma$ to $k\Gamma/I$ satisfying
$D(I)=\overline{0}$ can induce a differential operator of
$k\Gamma/I$.
\end{lemma}
\begin{proof}
Denote $p$ the canonical map from $k\Gamma$ to $k\Gamma/I$. Given a differential operator $D$ of $k\Gamma/I$, we claim that the composition $Dp$ is a differential operator from $k\Gamma$ to $k\Gamma/I$.  Note that the canonical map from $k\Gamma$ to $k\Gamma/I$ is an
algebra homomorphism, it can be directly verified. The converse result can be shown directly, too.
\end{proof}
For a differential operator $D$ from $k\Gamma$ to $k\Gamma/I$
satisfying $D(I)=\overline{0}$, we denote $\overline{D}$ the induced
differential operator on $k\Gamma/I$. Write
$$\mathfrak{F}(I):=\{D\ |\
D\in\textsl{Diff}(k\Gamma,k\Gamma/I),\ D(I)=\overline{0}\},\ \ \ \mathfrak{F}_i(I):=\{D\ |\
D\in\mathfrak{D}_i,\ D(I)=\overline{0}\}\ \text{for}\ i=1,2.$$
It is clear that $D_{\overline{s}}(I)=\overline{0}$ for $s\in
\mathscr{P}$. So $\mathfrak{F}_1(I)$=$\mathfrak{D}_1$ and $\mathfrak{F}(I)=\mathfrak{D}_1\oplus\mathfrak{F}_2(I)$.
\begin{lemma}\label{lemma 5.2}
$\mathfrak{F}(I)\cong \textsl{Diff}(k\Gamma/I)$  as $k$-linear spaces.
\end{lemma}
\begin{proof}
The map from $\mathfrak{F}(I)$ to $\textsl{Diff}(k\Gamma/I)$ is as follows,
$$\mathfrak{F}(I)\longrightarrow \textsl{Diff}(k\Gamma/I),\ D\longmapsto \overline{D}.$$
The proof of Lemma \ref{lemma 5.1} assures the map from $\mathfrak{F}(I)$ to
$\textsl{Diff}(k\Gamma/I)$ is surjective. As for the injectivity,
suppose $D_1, D_2\in\mathfrak{F}(I)$ and $D_1\neq D_2$, so according to Lemma \ref{lemma 2.2}, there
exists a path $p\in V\cup E$ such that $D_1(p)\neq D_2(p)$. Since
$\overline{0}\neq\overline{p}\in k\Gamma/I$, $\overline{D}_1(\overline{p})\neq
\overline{D}_2(\overline{p})$.
\end{proof}

By this lemma, we can think $\textsl{Diff}(k\Gamma/I)$  is a $k$-subspace of $\textsl{Diff}(k\Gamma,k\Gamma/I)$.

From Lemma \ref{lemma 5.2},
we have
\begin{equation}\label{4.29}
HH^1(k\Gamma/I)\cong\mathfrak{F}(I)/(\mathfrak{D}_1\oplus\mathfrak{D}_{C})\cong(\mathfrak{D}_1\oplus\mathfrak{F}_2(I))/(\mathfrak{D}_1\oplus\mathfrak{D}_{C})\cong\mathfrak{F}_2(I)/\mathfrak{D}_{C}
\end{equation} as linear spaces. This means
that $HH^1(k\Gamma/I)$ can be embedded into $H^1(k\Gamma,k\Gamma/I)\cong\mathfrak{D}_2/\mathfrak{D}_{C}$. Moreover, we have the next proposition.
\begin{proposition}\label{pro 4.3}
Suppose $k\Gamma/I$ is a finite dimensional admissible algebra, then
\begin{equation}\label{5.28}
dim_kHH^1(k\Gamma/I)=dim_k\mathfrak{F}_2(I)+dim_kZ(k\Gamma/I)-|\mathscr{Q}_C|.
\end{equation}
\end{proposition}
\begin{proof}
Note that $k\{\mathscr{Q}_C\}/(Z(k\Gamma/I))\cong\mathfrak{D}_{C}$ as linear spaces. By Equation $\ref{5.28}$, we have
$$dim_kHH^1(k\Gamma/I)=dim_k\mathfrak{F}_2(I)-dim_k\mathfrak{D}_{C}=dim_k\mathfrak{F}_2(I)+dim_kZ(k\Gamma/I)-|\mathscr{Q}_C|.$$
\end{proof}
\begin{corollary}\label{cor 4.4}
If $k\Gamma/I$ is an acyclic admissible algebra (in particular, if $\Gamma$ is an acyclic quiver), then
\begin{equation}
dim_kHH^1(k\Gamma/I)=dim_k\mathfrak{F}_2(I)+1-|V|.
\end{equation}
\end{corollary}
If $k\Gamma/I$ is an acyclic admissible algebra, we have a standard procedure to compute $dim_k\mathfrak{F}_2(I)$. First note that
for a differential operator $D$ from
$k\Gamma$ to $k\Gamma/I$, $D(I)=\overline{0}$ if and only if $D(r_i)=\overline{0}$ where $\{r_1,\cdots,r_i,\cdots r_n\}$ is a minimal set of generators of $I$. This
property follows easily from the Leibnitz rule of differential
operators. Since $k\Gamma/I$ is acyclic, $|\mathfrak{B}_2|$ is finite for $\mathfrak{B}_2$ as given in Theorem \ref{theorem 3.2}. Suppose $\sum\limits_{D_{r,\overline{s}}\in\mathfrak{B}_2}c_{r,\overline{s}}D_{r,\overline{s}}(r_i)=\overline{0}$ for $i=1,\cdots,n$. This means that the coefficients $c_{r,\overline{s}}$ satisfy the system of these homogeneous linear equations. So $dim\mathfrak{F}_2(I)$ is equal to the dimension of the solution space of the system of homogeneous linear equations.

Now we give two examples of admissible algebras that are not monomial algebras nor truncated quiver algebras, and characterize their first Hochschild cohomology.
\begin{example}
Let $\Gamma=(V,E)$ be the quiver \begin{picture}(70,20)
\put(35,20){\makebox(0,0){$\cdot$}}\put(35,-20){\makebox(0,0){$\cdot$}}
\put(65,0){\makebox(0,0){$\cdot$}}\put(5,0){\makebox(0,0){$\cdot$}}
\put(33,-18){\vector(-3,2){25}}\put(33,18){\vector(-3,-2){25}}
\put(63,-2){\vector(-3,-2){25}}\put(63,2){\vector(-3,2){25}}
\put(55,15){\makebox(0,0){$\alpha_1$}}\put(20,15){\makebox(0,0){$\alpha_2$}}
\put(55,-15){\makebox(0,0){$\beta_1$}}\put(20,-15){\makebox(0,0){$\beta_2$}}
\end{picture}
and $I=<\alpha_1\alpha_2-\beta_1\beta_2>$.

\

\

In this case, $\mathfrak{B}_2=\{D_{\alpha_1,\overline{\alpha_1}},D_{\alpha_2,\overline{\alpha_2}},D_{\beta_1,\overline{\beta_1}},D_{\beta_2,\overline{\beta_2}}\}$.
So we have
$$dim_kH^1(k\Gamma,k\Gamma/I)=|\mathfrak{B}_2|+dim_kZ(k\Gamma/I)-|\mathscr{Q}_C|=4+1-4=1.$$
Suppose that
$$(aD_{\alpha_1,\overline{\alpha_1}}+bD_{\alpha_2,\overline{\alpha_2}}+cD_{\beta_1,\overline{\beta_1}}+dD_{\beta_2,\overline{\beta_2}})
(\alpha_1\alpha_2-\beta_1\beta_2)=(a+b-c-d)\overline{\alpha\beta}=\overline{0},$$
then we get $a+b-c-d=0$. Hence  $dim_k\mathfrak{F}_2(I)=3$ and $dim_kHH^1(k\Gamma/I)=0$.
\end{example}

\begin{example}
Let $\Gamma$ be the quiver  having one vertex with two loops, equivalently,
$k\Gamma=k\langle x,y\rangle$. Suppose the ideal $I=\langle xy-yx\rangle$. Then $k\Gamma/I=k[x,y]$.

In this case, $\mathfrak{B}_1=\emptyset$ and
$\mathfrak{B}_2=\{D_{x,x^my^n}, D_{y,x^my^n}|m.n\geq0\}$ are the
basis of $\textsl{Diff}(k\langle x,y\rangle, k[x,y])$, where $x^my^n$ means the
multiplication in $k[x,y]$.

Since $k[x,y]$ is commutative, we get that
$$\textsl{Inn-Diff}(k\langle x,y\rangle, k[x,y])=0, \ \ \
H^1(k\langle x,y\rangle, k[x,y])=\textsl{Diff}(k\langle x,y\rangle, k[x,y]).$$
Moreover, note that $D_{x,x^my^n}(xy-yx)=0$ and
$D_{y,x^my^n}(xy-yx)=0$. Thus we obtain the basis of $HH^1(k[x,y])$ to be  $$\{D_{x,x^my^n},
D_{y,x^my^n}|m.n\geq0\}.$$

Similarly we can obtain the first Hochschild cohomology for $k[x_1,x_2,\cdots,x_n]$.
\end{example}
 Assume $k\Gamma/I$ is a monomial algebra. The residue classes of paths that do not belong to $I$ form a basis of $k\Gamma/I$. For convenience, we also denote by $\mathscr{Q}$ the basis of $k\Gamma/I$ when $k\Gamma/I$ is a monomial algebra.
\begin{definition}
A monomial algebra $k\Gamma/I$ is called {\bf complete} if for any parallel paths $p,p'$ in $\Gamma$, $p\in I$ implies $p'\in I$.
\end{definition}
\begin{proposition}\label{thm 4.5}
Suppose $k\Gamma/I$ is a complete monomial algebra with $I\subseteq R^2$.  Then the following set
\begin{equation}
\overline{\mathfrak{B}}\:=\overline{\mathfrak{B}}_1\cup\overline{\mathfrak{B}}_{2}
\end{equation}
is a basis of $\textsl{Diff}(k\Gamma/I)$,
where \begin{equation}
\overline{\mathfrak{B}}_1:=\{\overline{D}_{\overline{s}}\
|\overline{s}\in\mathscr{Q}, h(\overline{s})\neq t(\overline{s})\},
\ \ \ \overline{\mathfrak{B}}_{2}:=\{\overline{D}_{r,\overline{s}}\ |r\in
E,\overline{s}\in\mathscr{Q}, r\parallel \overline{s}\}.
\end{equation}
\end{proposition}
\begin{proof}
Since $k\Gamma/I$ is complete, we have $D_{r,\overline{s}}(p)=\overline{0}$ for any $D_{r,\overline{s}}\in\mathfrak{B}_2$, where $p$ is any path in $I$.
Then $\mathfrak{F}_2(I)=\mathfrak{D}_2$.  It follows that $\textsl{Diff}(k\Gamma,k\Gamma/I)\cong \textsl{Diff}(k\Gamma/I)$ as $k$-linear spaces. Thus due to Theorem \ref{theorem 3.2}, the result follows.
\end{proof}
\begin{corollary}
Suppose $k\Gamma/I$ is an acyclic complete monomial algebra with $I\subseteq R^2$. Then
\begin{equation}
dim_kHH^1(k\Gamma/I)=|\overline{\mathfrak{B}}_2|+1-|V|.
\end{equation}
\end{corollary}

\begin{proof}
By the proof of Proposition \ref{thm 4.5}, $dim_k\mathfrak{F}_2(I)=dim_k\mathfrak{D}_2=dim_k\overline{\mathfrak{D}}_2=|\overline{\mathfrak{B}}_2|$. By Corollary \ref{cor 4.4}, we get the required result.
\end{proof}

In \cite{Sa2}, the author gave a characterization of the first Hochschild cohomology of an acyclic complete monomial algebra through a projective resolution. However, its $k$-linear basis has not been constructed, so far. Here, we want to reach this aim in our method.
\begin{theorem}\label{tm 4.7}
Let $\Gamma$ be a planar quiver, $k\Gamma/I$ be an acyclic complete monomial algebra with $I\subseteq R^2$ over
the field $k$ of characteristic 0. Then the union set $$(\overline{\mathfrak{B}}_{2}\backslash\overline{\mathfrak{B}}_E)\cup\overline{\mathfrak{B}}_{\mathbb{P}}$$ is a basis of $HH^1(k\Gamma/I)$, where $\overline{\mathfrak{B}}_E=\{\overline{D}_{p,\overline{p}}|p\in E\}$ and $\overline{\mathfrak{B}}_{\mathbb{P}}=\{\overline{D}_{\mathbbm{p}}|\mathbbm{p}\in\Gamma_{\mathbb{P}}^{-}\}$.
\end{theorem}
\begin{proof}
By Eq.(\ref{4.29}) and $\mathfrak{F}_2(I)=\mathfrak{D}_2$, we have $HH^1(k\Gamma/I)\cong H^1(k\Gamma,k\Gamma/I)$ in this case. So from Theorem \ref{thm 3.7}, we can directly get this theorem.
\end{proof}

For a truncated quiver algebra $k\Gamma/k^n\Gamma$ with $n\geq2$, we can give a standard basis of
$\textsl{Diff}(k\Gamma/k^n\Gamma)$.
$k\Gamma/k^n\Gamma$ has the basis formed by
the residue  classes of the paths of length $\leq n-1$, denoted
also by $\mathscr{Q}$.
\begin{proposition}\label{thm 4.4}
Let $\Gamma=(V,E)$ be a quiver and the field $k$ be of
characteristic 0. A basis of $\textsl{Diff}(k\Gamma/k^n\Gamma)$ for any truncated quiver algebra $k\Gamma/k^n\Gamma$ with $n\geq 2$ is given by the set
\begin{equation}
\overline{\mathfrak{B}}\:=\overline{\mathfrak{B}}_1\cup\overline{\mathfrak{B}}_{2}
\end{equation}
where \begin{equation}
\overline{\mathfrak{B}}_1:=\{\overline{D}_{\overline{s}}\
|\overline{s}\in\mathscr{Q}, h(\overline{s})\neq t(\overline{s})\},\ \ \
\overline{\mathfrak{B}}_{2}:=\{\overline{D}_{r,\overline{s}}\ |r\in
E,\overline{s}\in\mathscr{Q}, s\notin V, r\parallel \overline{s}\}.
\end{equation}

\end{proposition}
\begin{proof} It is clear that $D_{\overline{s}}(k^n\Gamma)=\overline{0}$ for $\overline{s}\in\mathscr{Q}, h(\overline{s})\neq
t(\overline{s})$ and $D_{r,\overline{s}}(k^n\Gamma)=\overline{0}$
for $r\in E,\overline{s}\in\mathscr{Q}, s\notin V,
\overline{s}\parallel r$. Note that when $r$ is a loop of $\Gamma$,
$D_{r,\overline{h(r)}}\in \textsl{Diff}(k\Gamma,k\Gamma/k^n\Gamma)$, but
$D_{r,\overline{h(r)}}(r^n)=nr^{n-1}\neq \overline{0}$. Moreover,
for all loops $r_1,\cdots,r_s$ of $\Gamma$ and $c_1,\cdots,c_s$ not all $0$, we claim that $\sum\limits
c_iD_{r_i,\overline{h(r_i)}}(k^n\Gamma)\neq \overline{0}$. Without loss of generality, we can assume $c_1\neq 0$. So we have
$$\sum\limits
c_iD_{r_i,\overline{h(r_i)}}(r_1^n)=nc_1r_1^{n-1}\neq
\overline{0}.$$ Then by Theorem \ref{theorem 3.2}, the union set
$$\{D_{\overline{s}}|\overline{s}\in\mathscr{Q}, h(\overline{s})\neq
t(\overline{s})\}\bigcup\{D_{r,\overline{s}}|r\in
E,\overline{s}\in\mathscr{Q},\ s\notin V, r\parallel \overline{s}\}$$
forms a basis of the linear space $\mathfrak{F}_2(k^n\Gamma)$ for $I=k^n\Gamma$. By Lemma \ref{lemma 5.2}, we have
$$\mathfrak{F}_2(k^n\Gamma)\cong\textsl{Diff}(k\Gamma/k^n\Gamma).$$
Note the map from $\mathfrak{F}_2(k^n\Gamma)$ to $\textsl{Diff}(k\Gamma/k^n\Gamma)$ in Lemma \ref{lemma 5.2}, we can see that  the union set $\overline{\mathfrak{B}}\:=\overline{\mathfrak{B}}_1\cup\overline{\mathfrak{B}}_{2}$ is a $k$-linear basis of $\textsl{Diff}(k\Gamma/k^n\Gamma)$.
\end{proof}
Thus
$\textsl{Diff}(k\Gamma/k^n\Gamma)=\overline{\mathfrak{D}}_1\oplus
\overline{\mathfrak{D}}_{2}$, where $\overline{\mathfrak{D}}_i$ is
the $k$-linear space generated by $\overline{\mathfrak{B}}_i$ for
$i=1,2$.

\begin{corollary}\label{corollary 5.4}
Let $\Gamma=(V,E)$ be a quiver and the field $k$ be of
characteristic 0. Then
$$dim_kHH^1(k\Gamma/k^n\Gamma)=|\overline{\mathfrak{B}}_2|+dim_kZ(k\Gamma/k^n\Gamma)-|\mathscr{Q}_C|.$$
\end{corollary}
\begin{proof}
By the proof of Proposition \ref{thm 4.4} and the definition of $\mathfrak{F}_2(I)$, we can see that $\{D_{r,\overline{s}}\ |\ r\in
E,\overline{s}\in\mathscr{Q},\ s\notin V, r\parallel\overline{s}
\}$ is a basis of $\mathfrak{F}_2(k^n\Gamma)$ for $I=k^n\Gamma$. And by Proposition \ref{thm 4.4}, $\overline{\mathfrak{B}}_{2}:=\{\overline{D}_{r,\overline{s}}\ |r\in
E,\overline{s}\in\mathscr{Q}, s\notin V, r\parallel \overline{s}\}$.
 Then by Proposition \ref{pro 4.3} and the correspondence between $D_{r,\overline{s}}$ and $\overline{D}_{r,\overline{s}}$ for each pair  $(r,\overline{s})$, we get the required result.\end{proof}

This corollary has indeed been given as Theorem 1 in \cite{L} and Theorem 2 in \cite{XHJ}. The method we obtain it here is different with that in \cite{L} and \cite{XHJ}.

Moreover, when $k\Gamma/k^n\Gamma$ is acyclic, we can get a basis of
$HH^1(k\Gamma/k^n\Gamma)$ as in Theorem \ref{tm 4.7}.
\begin{theorem}\label{thm 5.3}
Let $\Gamma$ be a planar quiver, $k\Gamma/k^n\Gamma$ for $n\geq2$ be acyclic over
the  field $k$  of characteristic 0, then the union set
$$(\overline{\mathfrak{B}}_{2}\backslash\overline{\mathfrak{B}}_E)\cup\overline{\mathfrak{B}}_{\mathbb{P}}$$
is a basis of $HH^1(k\Gamma/k^n\Gamma)$, where $\overline{\mathfrak{B}}_E=\{\overline{D}_{p,\overline{p}}|p\in E\}$ and $\overline{\mathfrak{B}}_{\mathbb{P}}=\{\overline{D}_{\mathbbm{p}}|\mathbbm{p}\in\Gamma_{\mathbb{P}}^{-}\}$.
\end{theorem}
\begin{proof}
Since $k\Gamma/I$ is acyclic, $\mathfrak{D}_{C}=\mathfrak{D}_{V}$. By Lemma \ref{lemma 5.2}, $\overline{\mathfrak{D}}_{C}\cong\mathfrak{D}_{C}$, $\overline{\mathfrak{D}}_{E}\cong\mathfrak{D}_{E}$, $\overline{\mathfrak{D}}_{\mathbb{P}}\cong\mathfrak{D}_{\mathbb{P}}$. So the result can be obtained in the same way as the proof of Theorem \ref{thm 3.7}.
\end{proof}

\end{document}